\documentclass{svmult-ddm}

\usepackage{mathptmx}       
\usepackage{helvet}         
\usepackage{courier}        
\usepackage{type1cm}        
\usepackage{graphicx}        

\usepackage[bottom]{footmisc}

\usepackage{amsmath, amssymb}
\usepackage{amsfonts}
\usepackage{amsbsy}
\usepackage{amscd}
\usepackage{amstext}
\usepackage{dsfont}
\usepackage[english]{babel}
\usepackage{color}
\usepackage{graphics}
\usepackage{epsfig}
\usepackage{subfigure}
\usepackage{wrapfig}
\usepackage{psfrag}
\usepackage{color}
\usepackage{url}
\usepackage{verbatim}
\usepackage{float}

\begin{document}

\title*{A Time-Dependent Dirichlet-Neumann Method for the Heat Equation}
\titlerunning{Dirichlet-Neumann Waveform Relaxation}
\author{Bankim C. Mandal \inst{1}}
\institute{Department of Mathematics, University of Geneva, Geneva, Switzerland.
\texttt{Bankim.Mandal@unige.ch}
}
%

%
\maketitle
\abstract*{We present a waveform relaxation version of the Dirichlet-Neumann
method for parabolic problem. Like the Dirichlet-Neumann method for
steady problems, the method is based on a non-overlapping spatial
domain decomposition, and the iteration involves subdomain solves
with Dirichlet boundary conditions followed by subdomain solves with
Neumann boundary conditions. However, each subdomain problem is now
in space and time, and the interface conditions are also time-dependent.
Using a Laplace transform argument, we show for the heat equation
that when we consider finite time intervals, the Dirichlet-Neumann
method converges, similar to the case of Schwarz waveform relaxation algorithms.
The convergence rate depends on the length of the subdomains as well
as the size of the time window. In this discussion, we only stick
to the linear bound. We illustrate our results with numerical
experiments.}
\section{Introduction}
We introduce a new Waveform Relaxation (WR) method based on the Dirichlet-Neumann algorithm and present convergence results for it in one space dimension. To solve time-dependent problems in parallel, one can either discretize in time to obtain a sequence of steady problems to which the domain decomposition algorithms are applied, or apply WR to the large system of ordinary differential equations (ODEs) obtained from spatial discretization. The credit of WR method goes to Picard \cite{mandalb_mini15_Picard} and Lindel\" of \cite{mandalb_mini15_Lind} for the solution of ODEs in the late 19th century. Lelarasmee, Ruehli and Sangiovanni-Vincentelli \cite{mandalb_mini15_LelRue} were the first to introduce the WR as a parallel method for the solution of ODEs. The main advantage of the WR method is that one can use different time steps in different space-time subdomains. The authors of \cite{mandalb_mini15_GanStu} and \cite{mandalb_mini15_GilKel} then generalized WR methods for
ODEs to solve time-dependent PDEs. Gander and Stuart \cite{mandalb_mini15_GanStu} showed
linear convergence of overlapping Schwarz WR iteration for the heat equation
on unbounded time intervals with a rate depending on the size of the overlap;
Giladi and Keller \cite{mandalb_mini15_GilKel} proved superlinear convergence of the Schwarz WR method with overlap for the convection-diffusion equation on bounded time intervals. 

The Dirichlet-Neumann method, which belongs to the class of substructuring
methods, is based on a non-overlapping spatial domain decomposition.
The iteration involves subdomain solves with Dirichlet boundary conditions,
followed by subdomain solves with Neumann boundary conditions. The
Dirichlet-Neumann algorithm was first considered for elliptic problems
by P. E. Bj\o rstad \& O. Widlund \cite{mandalb_mini15_BjPetter} and further discussed in \cite{mandalb_mini15_BramPas},
\cite{mandalb_mini15_MarQuar02} and \cite{mandalb_mini15_MarQuar01}. In this paper, we propose the Dirichlet-Neumann Waveform Relaxation (DNWR) method, a new Dirichlet-Neumann analogue of WR for the time-dependent problems. 
For presentation purposes, we derive our results for two subdomains in one spatial dimension. We discuss the method in the continuous setting to ensure the understanding of the asymptotic behavior of the method in the case of fine grids.

We consider the following initial boundary value problem (IBVP) for the heat equation as our guiding example on a bounded domain $\Omega\subset\mathbb{R},0<t<T$, 
 \begin{equation}
\begin{array}{rcll}
\frac{\partial u}{\partial t} & = & \Delta u+f(x,t), & x\in\Omega,0<t<T,\\
u(x,0) & = & u_{0}(x), & x\in\Omega,\\
u(x,t) & = & g(x,t), & x\in\partial\Omega,0<t<T.
\end{array}
\label{mandalb_mini_eq:model}
\end{equation}
 
 \section{The Dirichlet-Neumann Waveform Relaxation algorithm}

To define the Dirichlet-Neumann iterative method for
the model problem (\ref{mandalb_mini_eq:model}) on the domain $(-b,a)\times(0,T)$, we split the spatial domain $\Omega=(-b,a)$ into two non-overlapping
subdomains, the Dirichlet subdomain $\Omega_{1}=(-b,0)$ and the Neumann subdomain
$\Omega_{2}=(0,a)$, for $0<a,b<\infty$. The Dirichlet-Neumann Waveform Relaxation
algorithm consists of the following steps: given an initial guess $h^{0}(t),t\in(0,T)$
along the interface $\Gamma=\left\{ x=0\right\} $ and for $k=0,1,2,\ldots$,
do
\begin{equation}\label{mandalb_mini_eq:dnalgo1}
\begin{cases}
\partial_{t}u_{1}^{k+1}-\partial_{xx}u_{1}^{k+1}=f(x,t), & x\in\Omega_{1},\\
u_{1}^{k+1}(x,0)=u_{0}(x),  &x\in\Omega_{1},\\
u_{1}^{k+1}(-b,t)=g(-b,t),\\
u_{1}^{k+1}(0,t)=h^{k}(t),
\end{cases}
\begin{cases}
\partial_{t}u_{2}^{k+1}-\partial_{xx}u_{2}^{k+1}=f(x,t), & x\in\Omega_{2},\\
u_{2}^{k+1}(x,0)=u_{0}(x), & x\in\Omega_{2},\\
\partial_{x}u_{2}^{k+1}(0,t)=\partial_{x}u_{1}^{k+1}(0,t),\\
u_{2}^{k+1}(a,t)=g(a,t),
\end{cases}
\end{equation}
with the updating condition 
\begin{equation}\label{mandalb_mini_eq:dnalgo2}
h^{k+1}(t)=\theta u_{2}^{k+1}(0,t)+(1-\theta)h^{k}(t),
\end{equation}
$\theta$ being a positive relaxation parameter. The parameter $\theta$
is chosen in $(0,1]$ to accelerate convergence. As the main goal of the analysis is to study how the error $h^{k}(t)-u(0,t)$ converges to zero, by linearity it suffices to consider the homogeneous
problem, $f(x,t)=0$, $g(x,t)=0$, $u_0(x)=0$ in (1), and examine how $h^{k}(t)$
goes to zero as $k\rightarrow\infty$. 

\section{Convergence analysis and main results}

We analyze the DNWR algorithm using the Laplace transform method. The Laplace transform of a function $w(t)$, defined for all real numbers $t\in[0,\infty)$, is the function $\hat{w}(s)$, defined by
\[
\hat{w}(s)=\mathcal{L}\left\{ w(t)\right\} :=\int_{0}^{\infty}e^{-st}w(t)\, dt,
\]
(if the integral exists) $s$ being a complex variable. If $\mathcal{L}\left\{ w(t)\right\} =\hat{w}(s)$, then the inverse Laplace transform of $\hat{w}(s)$ is denoted by
\[
\mathcal{L}^{-1}\left\{ \hat{w}(s)\right\} :=w(t),\qquad t\geq0,
\]
which maps the Laplace transform of a function back to the original function. For more information on Laplace transforms, see \cite{mandalb_mini15_Church,mandalb_mini15_Oberhett}. We use hats to denote the Laplace transform of a function in time in the rest of the paper.
\vspace{0.2cm}\\
{\bf Analysis by Laplace transforms.} Applying a Laplace transform in time to (\ref{mandalb_mini_eq:dnalgo1}) and solving the
resulting ODEs yields the solutions:
 $\hat{u}_{1}^{k+1}(x,s)=\frac{\hat{h}^{k}(s)}{\sinh(b\sqrt{s})}\sinh\left\{ (x+b)\sqrt{s}\right\} $
and $\hat{u}_{2}^{k+1}(x,s)=\hat{h}^{k}(s)\frac{\coth(b\sqrt{s})}{\cosh(a\sqrt{s})}\sinh\{(x-a)\sqrt{s}\}.$
Now, evaluating $\hat{u}_{2}^{k+1}(x,s)$ at $x=0$ and inserting
it into the transformed updating condition (\ref{mandalb_mini_eq:dnalgo2}), we get for $k=0,1,2,\ldots$
$\hat{h}^{k+1}(s)=\left\{ 1-\theta-\theta\tanh(a\sqrt{s})\coth(b\sqrt{s})\right\} \hat{h}^{k}(s).$
Therefore, by induction we get
\begin{equation}\label{mandalb_mini_eq:update1}
\hat{h}^{k}(s)=\left\{ 1-\theta-\theta\tanh(a\sqrt{s})\coth(b\sqrt{s})\right\} ^{k}\hat{h}^{0}(s),\quad k=1,2,3,...
\end{equation}
\begin{theorem}
For the symmetric case, $a=b$ in (\ref{mandalb_mini_eq:dnalgo1})-(\ref{mandalb_mini_eq:dnalgo2}), the DNWR algorithm
converges linearly for $0<\theta<1$. Moreover, for $\theta=0.5$, it converges to the
exact solution in two iterations, independent of the size of the time window.\end{theorem}
\begin{proof}
For $a=b$, the equation (\ref{mandalb_mini_eq:update1}) reduces
to $\hat{h}^{k}(s)=(1-2\theta)^{k}\hat{h}^{0}(s),$ which upon back
transforming gives $h^{k}(t)=(1-2\theta)^{k}h^{0}(t)$. Thus, the convergence is linear for $\theta\neq0.5$. On the other hand, for $\theta=0.5$, $h^{1}(t)=0$. Therefore, one more iteration produces the desired solution on the whole domain.$\hfill\qed$
\end{proof}

The main area of concern for the rest of the paper is the analysis of the DNWR algorithm for $a\neq b$. If we define 
$$G(s):=\tanh(a\sqrt{s})\coth(b\sqrt{s})-1=\frac{\sinh((a-b)\sqrt{s})}{\cosh(a\sqrt{s})\sinh(b\sqrt{s})},$$
then the recurrence relation (\ref{mandalb_mini_eq:update1}) reduces to
\begin{equation}\label{mandalb_mini_eq:updatefinal}
\hat{h}^{k}(s)=\begin{cases}
\left\{ q(\theta)-\theta G(s)\right\} ^{k}\hat{h}^{0}(s), & \theta\neq1/2\\
(-1)^{k}2^{-k}G^{k}(s)\hat{h}^{0}(s), & \theta=1/2,
\end{cases}
\end{equation}
where $q(\theta)=1-2\theta.$ Note that for $\textrm{Re}(s)>0,$ $G(s)$
is\footnote {Assuming $s=re^{i\vartheta}$, $z=\sqrt{s}$, we can write for $b\geq a$,
$\bigl|s^{p}G(s)\bigr|\leq\bigl|\frac{s^{p}}{\cosh(az)}\bigr|\leq\frac{2r^{p}}{|e^{a\sqrt{r/2}}-e^{-a\sqrt{r/2}}|}\rightarrow0$,
as $r\rightarrow\infty$; and for $a>b$, $\bigl|s^{p}G(s)\bigr|\leq\bigl|\frac{s^{p}}{\sinh(bz)}\bigr|\leq\frac{2r^{p}}{|e^{b\sqrt{r/2}}-e^{-b\sqrt{r/2}}|}\rightarrow0$, as $r\rightarrow\infty$.%
} $\mathcal{O}(s^{-p})$ for every positive $p$. Therefore, by \cite[p.~178]{mandalb_mini15_Church}, $G(s)$ is
the Laplace transform of an analytic function $F_{1}(t)$ (in fact this
is the motivation in defining $G$). In general, define $F_{k}(t):=\mathcal{L}^{-1}\left\{ G^{k}(s)\right\} $ for $k=1,2,3,\ldots$. For $\theta$ not equal to $1/2$, $h^{k}$ cannot be expressed as a simple convolution of $h^{0}$ and an analytic function; thus, different techniques are required to analyze its behavior. This case will be treated in a future paper. 
 For $\theta=1/2$ and $t\in(0,T)$ we get from (\ref{mandalb_mini_eq:updatefinal})
\begin{equation}\label{mandalb_mini_eq:boundstep}
\bigl|h^{k}(t)\bigr|=\biggl|2^{-k}\int_{0}^{t}(-1)^{k}h^{0}(t-\tau)F_{k}(\tau)d\tau\biggr|\leq2^{-k}\parallel h^{0}\parallel_{L^{\infty}(0,T)}\int_{0}^{T}\Bigl|F_{k}(\tau)\Bigr|d\tau.
\end{equation}
So, we need to bound $\int_{0}^{T}\bigl|F_{k}(\tau)\bigr|d\tau$
to get an $L^{\infty}$ convergence estimate. We concentrate
on showing that $F_{1}(t)$ does not change signs both for the case $b<a$, in which $F_{1}(t)\geq0$, and for $b\geq a$, for which $F_{1}(t)\leq0$. Before we proceed further with the proof we need the following lemmas.
\begin{lemma}
Let, $w(t)$ be a continuous and $L^{1}$-integrable function on $(0,\infty)$
with $w(t)\geq0$ for all $t\geq0$. Assume $W(s)=\mathcal{L}\left\{ w(t)\right\} $.
Then, for $\tau>0,$
\[
\int_{0}^{\tau}|w(t)|dt\leq{\displaystyle \lim_{s\rightarrow0+}}W(s).
\]
\end{lemma}
\begin{proof}
Using the definition of Laplace transform, we have
\begin{align*}
\int_{0}^{\tau}|w(t)|dt =\int_{0}^{\tau}w(t)dt &\leq\int_{0}^{\infty}w(t)dt\\
=\int_{0}^{\infty}{\displaystyle \lim_{s\rightarrow0+}}e^{-st}w(t)dt&={\displaystyle \lim_{s\rightarrow0+}}\int_{0}^{\infty}e^{-st}w(t)dt \;\mbox{(by Dominated Conv. Theorem)}\\
&={\displaystyle \lim_{s\rightarrow0+}}W(s).\quad\quad\quad\quad\quad\quad\quad\qed
\end{align*}
\end{proof}
\begin{lemma}
Let $\beta>\alpha\geq0$ and $s$ be a complex variable. Then, for
$t\in(0,\infty)$ 
\[
\varphi(t):=\mathcal{L}^{-1}\left\{ \frac{\sinh(\alpha\sqrt{s})}{\sinh(\beta\sqrt{s})}\right\} \geq0\;;\;\:\psi(t):=\mathcal{L}^{-1}\left\{ \frac{\cosh(\alpha\sqrt{s})}{\cosh(\beta\sqrt{s})}\right\} \geq0.
\]
\end{lemma}
\begin{proof}
First, let us consider the following IBVP for the heat equation on $(0,\beta)$: $u_{t}-u_{xx}=0,\, u(x,0)=0,\: u(0,t)=0,\: u(\beta,t)=g(t).$ Therefore, for $g$ non-negative, $u(\alpha,t)$ is also non-negative for all
$t>0$, thanks to the maximum principle. Now using the Laplace transform method, we get the solution along $x=\alpha$
as 
\[
\hat{u}(\alpha,s)=\hat{g}(s)\frac{\sinh(\alpha\sqrt{s})}{\sinh(\beta\sqrt{s})}\quad\Longrightarrow\quad u(\alpha,t)=\int_{0}^{t}g(t-\tau)\varphi(\tau)d\tau.
\]
We prove the result by contradiction: suppose $\varphi(t_{0})<0$ for some $t_{0}>0$. Then
by continuity of $\varphi$, there exists $\delta>0$ such that $\varphi(\tau)<0,$
for $\tau\in(t_{0}-\delta,t_{0}+\delta)$. Now for $t>t_0+\delta$, we choose $g$ as 
\[
g(\varsigma)=\begin{cases}
1, & \varsigma\in\left(t-t_{0}-\delta,t-t_{0}+\delta\right)\\
0, & \textrm{else}.
\end{cases}
\]
Then $u(\alpha,t)=\int_{t_{0}-\delta}^{t_{0}+\delta}g(t-\tau)\varphi(\tau)d\tau=\int_{t_{0}-\delta}^{t_{0}+\delta}\varphi(\tau)d\tau<0$,
a contradiction. This proves $\varphi$ to be non-negative. For $\psi$, applying the Laplace transform method to the IBVP for the heat equation $u_{t}-u_{xx}=0,\: u(x,0)=0,\: u(-\beta,t)=g(t),\: u(\beta,t)=g(t)$ yields the solution along $x=\alpha$ as: $\hat{u}(\alpha,s)=\hat{g}(s)\frac{\cosh(\alpha\sqrt{s})}{\cosh(\beta\sqrt{s})}$. Thus, a similar argument as in the first case proves that $\psi$ is also non-negative.$\hfill\qed$ 
\end{proof}

\begin{theorem}
\textbf{(Linear convergence bound for the Heat equation)} Let $\theta=1/2$. For
$T>0$, the error of the Dirichlet-Neumann Waveform Relaxation (DNWR)
algorithm satisfies 
\[
\parallel h^{k}\parallel_{L^{\infty}(0,T)}\leq\left(\frac{|b-a|}{2b}\right)^{k}\parallel h^{0}\parallel_{L^{\infty}(0,T)}.
\] We therefore have a contraction if $a<3b.$
\end{theorem}
\begin{proof}
By virtue of (\ref{mandalb_mini_eq:boundstep}), it is sufficient to bound $\int_{0}^{T}\bigl|F_{k}(\tau)\bigr|d\tau$ for both $b\geq a$ and $a>b$, where $F_{k}(t)=\mathcal{L}^{-1}\left\{ G^{k}(s)\right\} $.
Suppose $b\geq a>0.$ We have $\mathcal{L}\left\{-F_1(t)\right\} =\frac{\sinh((b-a)\sqrt{s})}{\sinh(b\sqrt{s})}\cdot \frac{1}{\cosh(a\sqrt{s})}.$
So by Lemma 2 and the fact that the convolution
of two positive functions is positive, $-F_{1}(t)$ is positive. Thus, by induction and with the same arguments,
$(-1)^{k}F_{k}(t)\geq0$ for all $t$. Therefore by Lemma 1
\begin{equation}\label{mandalb_mini_eq:linbnd1}
\int_{0}^{T}\bigl|(-1)^{k}F_{k}(\tau)\bigr|d\tau\leq{\displaystyle \lim_{s\rightarrow0+}}(-1)^{k}G^{k}(s)=\left(\frac{b-a}{b}\right)^{k}.
\end{equation}
Now let $a>b>0$. We claim that $F_{1}(t)$ is positive.
If $a-b\leq b$, then we get the positivity by Lemma 2. If
this is not the case, then take the integer $m=\lfloor a/b \rfloor$ so that $mb<a\leq(m+1)b$. Then, recursively applying the identity  
$$\frac{\sinh((a-jb)\sqrt{s})}{\sinh(b\sqrt{s})}=\frac{\sinh((a-(j+1)b)\sqrt{s})}{\sinh(b\sqrt{s})}\cosh(b\sqrt{s})+\cosh((a-(j+1)b)\sqrt{s})$$ for $j=1,\ldots,m-1$, we obtain
\begin{multline*}
\frac{\sinh((a-b)\sqrt{s})}{\cosh(a\sqrt{s})\sinh(b\sqrt{s})}=\frac{\sinh((a-mb)\sqrt{s})}{\sinh(b\sqrt{s})}.\frac{\cosh^{m-1}(b\sqrt{s})}{\cosh(a\sqrt{s})}\\
+{\displaystyle \sum_{j=0}^{m-2}}\frac{\cosh^{j}(b\sqrt{s})\cosh\left((a-(j+2)b)\sqrt{s}\right)}{\cosh(a\sqrt{s})}.
\end{multline*}
Applying the binomial theorem to $\cosh\theta=\left(e^{\theta}+e^{-\theta}\right)/2$
we have the power-reduction formula 
\[
\cosh^{n}\theta=\begin{cases}
\frac{2}{2^{n}}{\displaystyle \sum_{l=0}^{\frac{n-1}{2}}}\binom{n}{l}\cosh\left((n-2l)\theta\right), & n\:\:\textrm{odd},\\
\frac{1}{2^{n}}\binom{n}{n/2}+\frac{2}{2^{n}}{\displaystyle \sum_{l=0}^{\frac{n}{2}-1}}\binom{n}{l}\cosh\left((n-2l)\theta\right), & n\:\:\textrm{even},
\end{cases}
\]
so that we can write $\cosh^{n}\theta={\displaystyle \sum_{l=0}^{n}}A_{l}^{n}\cosh\left(l\theta\right)$
with ${\displaystyle \sum_{l=0}^{n}}A_{l}^{n}=1$ and $A_{l}^{n}\geq 0$. Therefore, we have
\begin{multline*}
G(s)=\frac{\sinh((a-b)\sqrt{s})}{\cosh(a\sqrt{s})\sinh(b\sqrt{s})}=\frac{\sinh((a-mb)\sqrt{s})}{\sinh(b\sqrt{s})}{\displaystyle \sum_{l=0}^{m-1}}A_{l}^{m-1}\frac{\cosh(lb\sqrt{s})}{\cosh(a\sqrt{s})}\\
+{\displaystyle \sum_{j=0}^{m-2}}{\displaystyle \sum_{l=0}^{j}}\frac{A_{l}^{j}}{2}\left\{ \frac{\cosh\left((a-(j+l+2)b)\sqrt{s}\right)}{\cosh(a\sqrt{s})}+\frac{\cosh\left((a-(j-l+2)b)\sqrt{s}\right)}{\cosh(a\sqrt{s})}\right\} ,
\end{multline*}
where $\cosh^{j}\theta={\displaystyle \sum_{l=0}^{j}}A_{l}^{j}\cosh\left(l\theta\right)$. Note that $a-mb\leq b$, $(j-l+2)b\leq mb<a$ and $\left|a-(j+l+2)b\right|< a$
for $0\leq j,l\leq m-2$ and $\cosh$ is an even function. Thus by Lemma 2, each
term in the above expression is the Laplace transform of a positive
function. Hence $F_{1}(t)$ is positive, and therefore the convolution of $k$ $F_{1}$'s (i.e. $F_{k}(t)$) is also positive.
We have ${\displaystyle \lim_{s\rightarrow0+}}G(s)={\displaystyle \lim_{s\rightarrow0+}}\frac{\sinh((a-b)\sqrt{s})}{\cosh(a\sqrt{s})\sinh(b\sqrt{s})}=\frac{a-b}{b},$ and so by Lemma 1 
\begin{equation}\label{mandalb_mini_eq:linbnd2}
\int_{0}^{T}\bigl|F_{k}(\tau)\bigr|d\tau=\int_{0}^{T}F_{k}(\tau)d\tau\leq{\displaystyle \lim_{s\rightarrow0+}}G^{k}(s)=\left({\displaystyle \lim_{s\rightarrow0+}}G(s)\right)^{k}=\left(\frac{a-b}{b}\right)^{k}.
\end{equation}
The result follows by inserting the estimates (\ref{mandalb_mini_eq:linbnd1}) and (\ref{mandalb_mini_eq:linbnd2}) into (\ref{mandalb_mini_eq:boundstep}). $\hfill\qed$
\end{proof}
\section{Numerical Experiments}
We perform experiments to measure the actual convergence rate of the DNWR algorithm for the problem
\[
\begin{cases}
\frac{\partial u}{\partial t}-\frac{\partial^{2}u}{\partial x^{2}}=-e^{-t-x^{2}}, & x\in(-3,2),\\
u(x,0)=e^{-2x}, & x\in(-3,2),\\
u(-3,t)=e^{-2t}=u(2,t), & t>0.
\end{cases}
\] 
To solve the equation using the Dirichlet-Neumann algorithm, we discretize 
the Laplacian using centered finite differences in space and backward Euler
in time on a grid with $\Delta x=2\times 10^{-2}$
and $\Delta t=4\times 10^{-4}$. For the numerical experiments
we split the spatial domain into two non-overlapping subdomains $\left[-3,0\right]$
and $\left[0,2\right]$, so that $b=3$ and $a=2$ in
(\ref{mandalb_mini_eq:dnalgo1})-(\ref{mandalb_mini_eq:dnalgo2}). Thus this is the case when the Dirichlet subdomain is bigger than the Neumann subdomain. The numerical results are similar for the case when the Neumann domain is larger than the Dirichlet one. We test the algorithm by choosing  $h^{0}(t)=t$, $t\in (0,T]$ as an initial guess. Figure 1 gives the error reduction curves for different values of the parameter $\theta$ for $T=2$ in $(a)$ and $T=200$ in $(b)$. Note that, for a small time window, we get linear convergence for all the parameters, except for $\theta=0.5$ which corresponds to superlinear convergence. 
\begin{figure}[H]\begin{centering}
\mbox{\subfigure[Short time window]{\includegraphics[width=5.8cm,height=3.5cm]{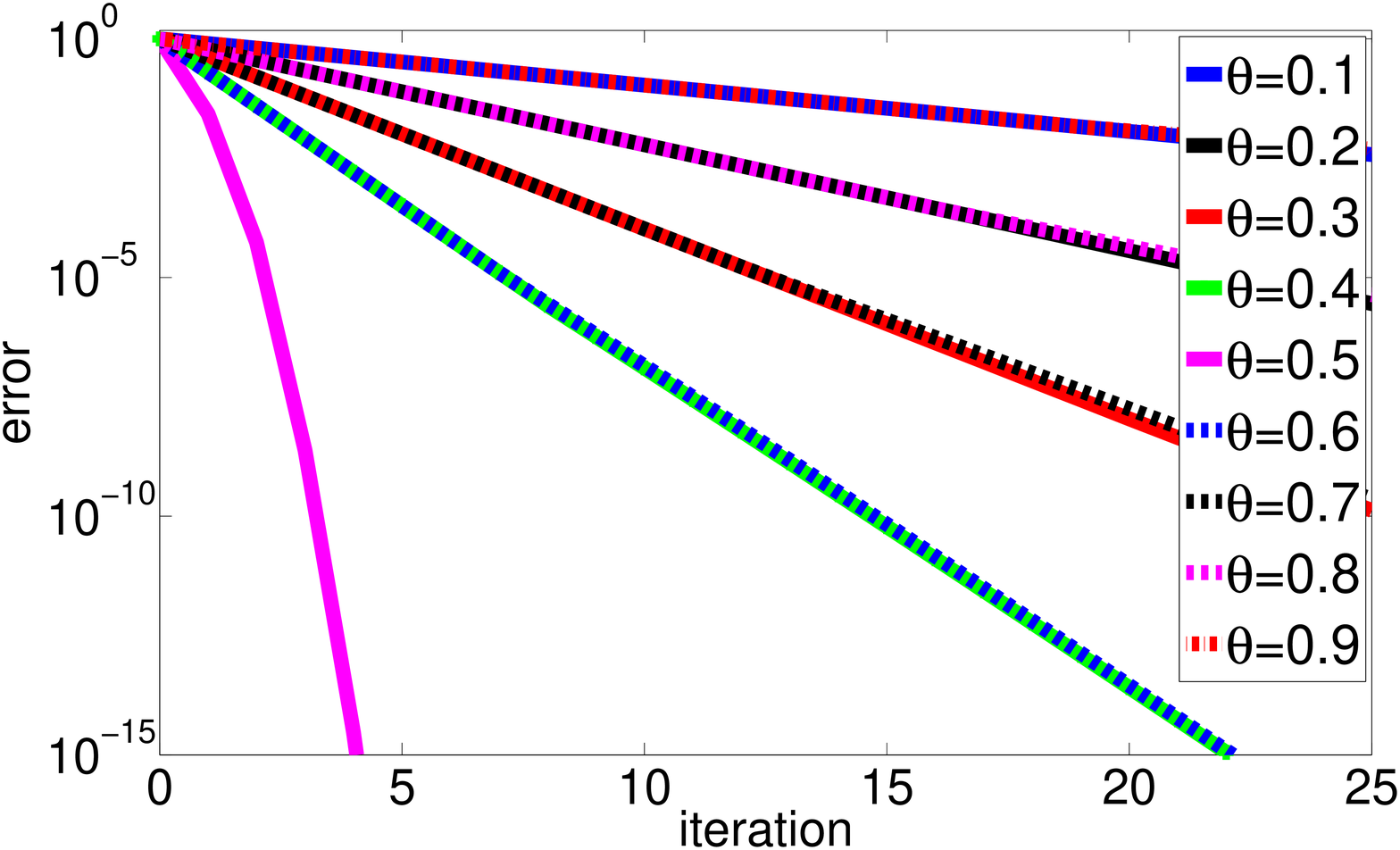}}\quad
\subfigure[Large time window]{\includegraphics[width=5.8cm,height=3.5cm]{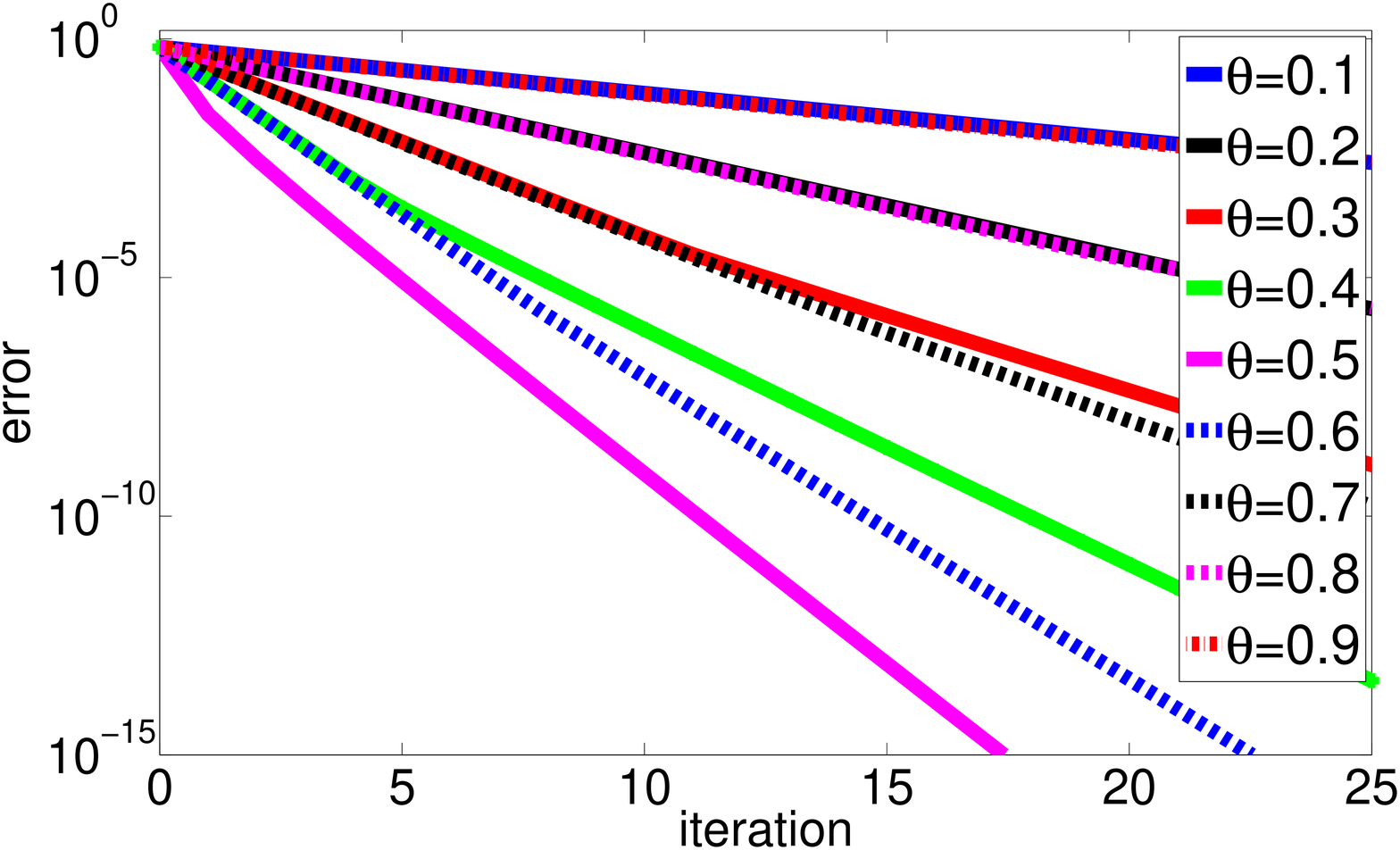}}}
\par\end{centering}
\caption{Convergence for various parameters; left: short time window, right: large time window.}
\end{figure}
\noindent For a large time window, we always observe linear convergence. We now plot the linear bound for the convergence rate in case of $\theta=1/2$ as shown in Theorem 2. The theorem provides a $T$-independent theoretical bound of the error for this special relaxation parameter and this is also valid for large time windows. Eventually, a more refined analysis will give a superlinear bound shown in (\ref{mandalb_mini_eq:supbnd1})-(\ref{mandalb_mini_eq:supbnd2}), dependent on $T$ and the lengths of the subdomains (see \cite{mandalb_mini15_GKM}). Figure 2 gives a comparison between the theoretical error for the continuous model problem (calculated using inverse Laplace transforms), numerical error for the
discretized problem, linear bound and the superlinear bound for $a=2,b=3$ and various $T$'s. We can observe that the error curves seem to approach the linear bound as $T$ increases.
 \begin{figure}[H]\begin{centering}
\mbox{\subfigure{\includegraphics[width=6.1cm,height=4.1cm]{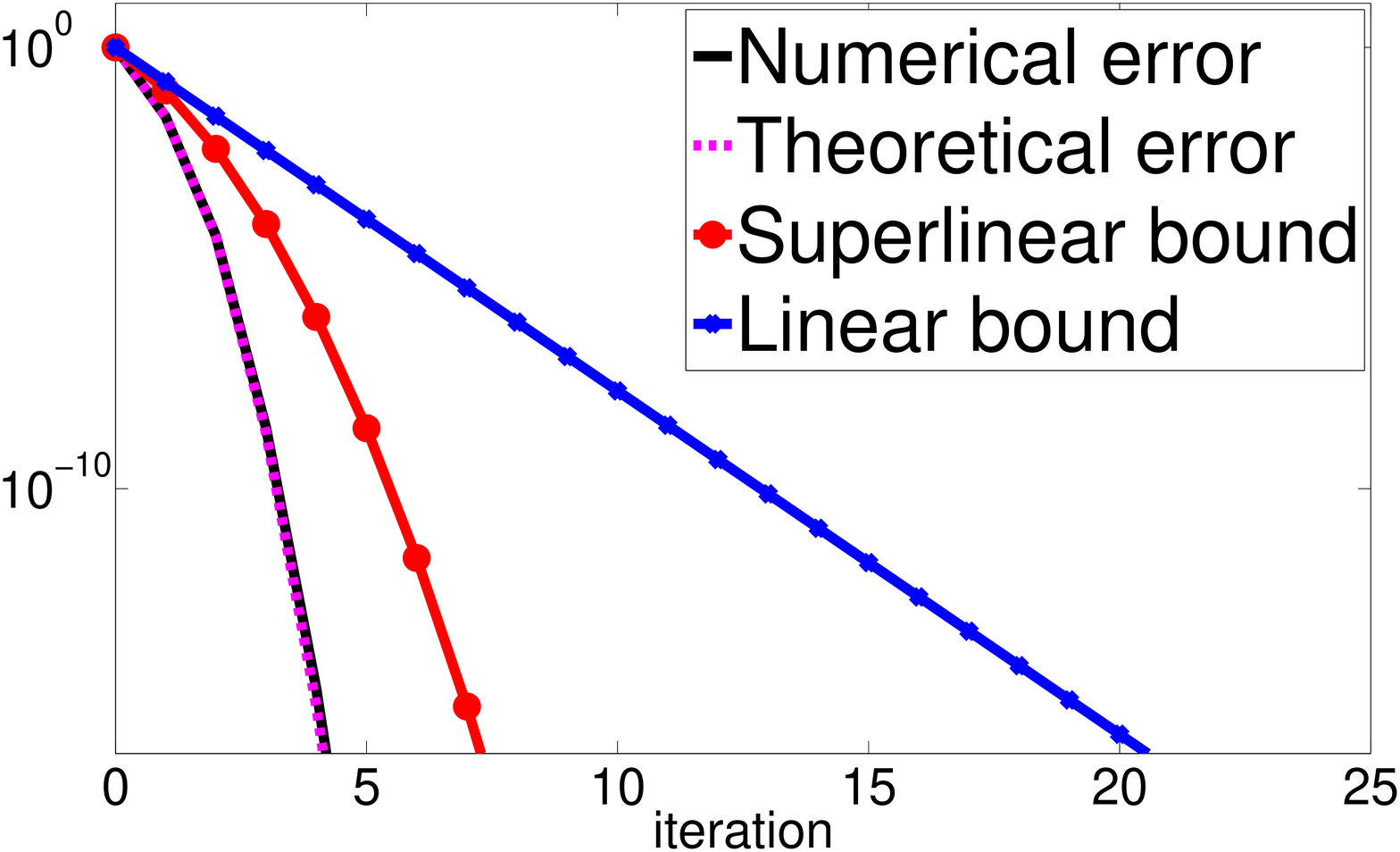}}\quad
\subfigure{\includegraphics[width=6.1cm,height=4.1cm]{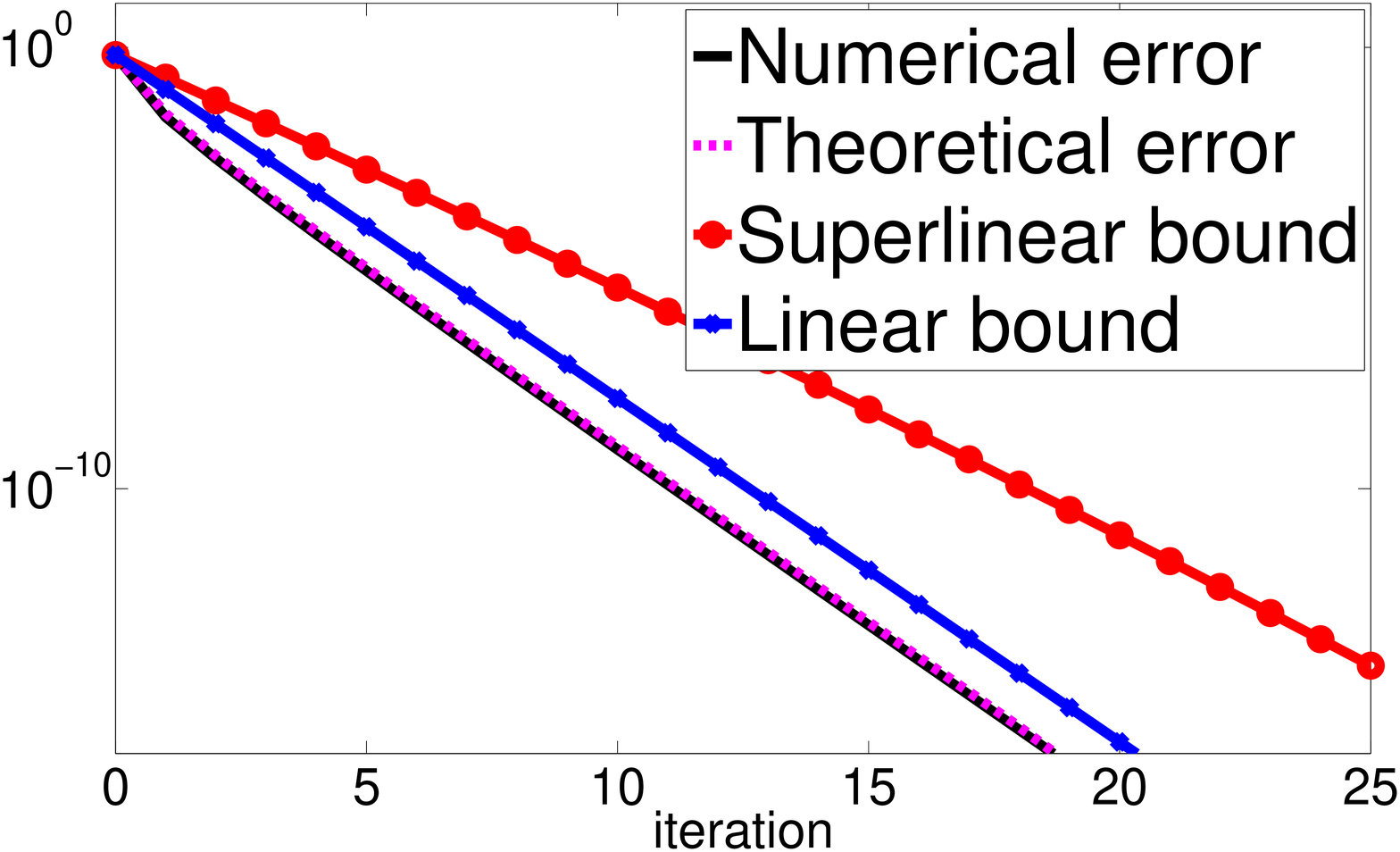}}}
\par\end{centering}
\caption{Bounds for various times, $b\geq a$; in particular $a<3b$. Left: $T=2$, right: $T=200$ }
\end{figure}

\section{Conclusions and further results}
We proved convergence of the proposed DNWR algorithm in the symmetric case. For unequal
subdomain lengths and for a particular choice of relaxation parameter, we presented a linear error estimate that is valid for both bounded and unbounded time intervals. In fact, Figure 2 suggests that the method converges superlinearly. 
\begin{wrapfigure}{r}{0.5\textwidth}
  \vspace{-20pt}
  \begin{center}
    \includegraphics[width=0.5\textwidth]{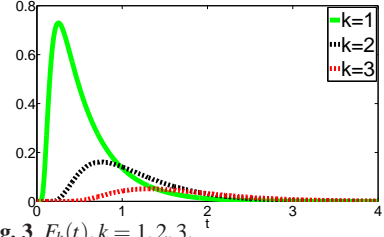}
  \end{center}
  \vspace{-20pt}
  \centering \caption{$F_k(t), k=1,2,3.$}
  \vspace{-10pt}
\end{wrapfigure} 
 To prove this, one has to consider two different cases: Dirchlet subdomain bigger than Neumann subdomain ($b\geq a$) and the other way around. Figure 3 shows $F_{k}(t)$ for $k=1,2,3$; we see that the curves shift to the right and at the same time, the peak decreases as $k$ increases. So, if one only considers a small time window, the peak will eventually exit the time window for $k$ large enough and its contribution will be vanishingly small in the expression (\ref{mandalb_mini_eq:updatefinal}). This is the intuitive idea to get superlinear convergence for $\theta=1/2$ in small time windows. A detailed analysis, which is too long for this short paper, in \cite{mandalb_mini15_GKM} leads to the following superlinear convergence estimates for the small time window $(0,T)$:
\begin{equation}\label{mandalb_mini_eq:supbnd1}
\parallel h^{k}\parallel_{L^{\infty}(0,T)}\leq\left(\frac{b-a}{b}\right)^{k}\textrm{erfc}\left(\frac{ka}{2\sqrt{T}}\right)\parallel h^{0}\parallel_{L^{\infty}(0,T)}, \quad \mbox{for}\; b\geq a,
\end{equation}
and \begin{equation}\label{mandalb_mini_eq:supbnd2}\parallel h^{2k}\parallel_{L^{\infty}(0,T)}\leq\left\{ \frac{\sqrt{2}}{1-e^{-\frac{2k+1}{\sigma}}}\right\} ^{2k}e^{-k^{2}/\sigma}\parallel h^{0}\parallel_{L^{\infty}(0,T)}, \quad \mbox{for} \; b<a,
\end{equation}
where $\sigma=T/b^{2}.$ We are also working on a generalization of the algorithm to higher dimensions.
\begin{acknowledgement}
I would like to express my gratitude to Prof. Martin J. Gander and Dr. Felix Kwok for their constant support and stimulating suggestions.
\end{acknowledgement}

\nocite{*}


\end{document}